\documentclass[a4paper,12pt]{article}
\usepackage[utf8]{inputenc}
\usepackage[english]{babel}
\usepackage{amsmath}
\usepackage{amsfonts}
\usepackage{amsthm}
\usepackage{amssymb}
\usepackage{mathrsfs}
\usepackage{graphicx}
\usepackage{float}
\usepackage{enumerate}
\usepackage{hyperref}
\usepackage[margin=3cm]{geometry}
\usepackage{tikz-cd}
\usetikzlibrary{decorations.pathmorphing}

\begin{document}

	\title{On the classification of rational four-dimensional unital division algebras}
	\author{Gustav Hammarhjelm}
	\maketitle
		
	
	\newcommand{\ad}[1]{\mathrm{ad}_{#1}}
	\newcommand{\Ad}[1]{\mathrm{Ad}_{#1}}
	\newcommand{\ASL}[2]{\mathrm{ASL}(#1,#2)}
	\newcommand{\Aut}[2]{\mathrm{Aut}_{#1}(#2)}
	\newcommand{\bb}[1]{\mathbb{#1}}
	\newcommand{\ceil}[1]{\lceil#1\rceil}
	\newcommand{\Char}{\mathrm{char}\,}
	\newcommand{\complex}[1]{(#1_n,\partial_n)_{n\in\mathbb{Z}}}
	\newcommand{\cyclic}[1]{\langle #1\rangle}
	\newcommand{\Der}[1]{\mathrm{Der}(C^\infty(#1))}
	\newcommand{\dirlim}{\varinjlim}
	\newcommand{\End}[2]{\mathrm{End}_{#1}(#2)}
	\newcommand{\Ext}[1]{\mathrm{Ext}^{#1}}
	\newcommand{\floor}[1]{\lfloor#1\rfloor}
	\newcommand{\fr}[1]{\mathfrak{#1}}
	\newcommand{\Gal}[2]{\mathrm{Gal}(#1/#2)}
	\newcommand{\GL}[2]{\mathrm{GL}(#1,#2)}
	\newcommand{\Hom}{\mathrm{Hom}}
	\newcommand{\id}{\mathrm{id}}
	\newcommand{\image}{\mathrm{im}\,}
	\newcommand{\indexset}[1]{\{1,\ldots,#1\}}
	\newcommand{\inner}[2]{\langle #1,#2\rangle}
	\newcommand{\Int}[1]{\mathrm{Int}\left(#1\right)}
	\newcommand{\inv}[1]{#1^{-1}}
	\newcommand{\invlim}{\varprojlim}
	\newcommand{\leg}[2]{\left(\frac{#1}{#2}\right)}
	\newcommand{\liminfl}[1]{\liminf_{#1\to\infty}}
	\newcommand{\limsupl}[1]{\limsup_{#1\to\infty}}
	\newcommand{\norm}[1]{\left\| #1\right\|}
	\newcommand{\normp}[1]{\left\| #1\right\|_p}
	\newcommand{\pderiv}[2]{\frac{\partial #1}{\partial #2}}
	\newcommand{\rad}{\mathrm{rad}\,}
	\newcommand{\scr}[1]{\mathscr{#1}}
	\newcommand{\SES}[5]{0\longrightarrow#1\overset{#4}{\longrightarrow}#2\overset{#5}{\longrightarrow}#3\longrightarrow 0}
	\newcommand{\sgn}[1]{\mathrm{sgn}\left(#1\right)}
	\newcommand{\SL}[2]{\mathrm{SL}(#1,#2)}
	\newcommand{\Span}[1]{\mathrm{span}_{#1}}
	\newcommand{\smpt}[1]{\setminus\{#1\}}
	\newcommand{\Tor}[1]{\mathrm{Tor}_{#1}}
	\newcommand{\tr}[1]{\mathrm{tr}(#1)}
	\newcommand{\vol}{\mathrm{vol}}
	
		
	\newtheorem{thm}{Theorem}[section]
	\newtheorem*{thm*}{Theorem}
	\newtheorem{clm}[thm]{Claim}
	\newtheorem{fact}[thm]{Fact}
	\newtheorem{lem}[thm]{Lemma}
	\newtheorem{cor}[thm]{Corollary}
	\newtheorem{prop}[thm]{Proposition}
	\newtheorem{conj}[thm]{Conjecture}
	\theoremstyle{definition}
	\newtheorem{defn}[thm]{Definition}
	\newtheorem*{ex}{Example}
	\theoremstyle{remark}
	\newtheorem*{rem}{Remark}
	\newtheorem*{exer}{Exercise}
	\theoremstyle{remark}
	\newtheorem*{sol}{Solution}
	\theoremstyle{remark}
	
	\begin{abstract}
		In \cite{dieterich2017onFour}, the category $\mathscr{C}(k)$ of four-dimensional unital division algebras, whose right nucleus is non-trivial and whose automorphism group contains Klein's four group $V$, is studied over a general ground field $k$ with $\Char k\neq 2$. In particular, the objects in $\mathscr{C}(k)$ are exhaustively constructed from parameters in $k^3$ and explicit isomorphism conditions for the constructed objects are found in terms of these parameters.
			
		In this paper, we specialize to the case $k=\bb{Q}$ and present results towards a classification of $\mathscr{C}(\bb{Q})$. In particular, for each field $\ell$ with $[\ell:k]=2$ we present explicity a two-parameter family of pairwise non-isomorphic non-associative objects in $\mathscr{C}(\bb{Q})$ that admit $\ell$ as a subfield and we provide a method for classifying the full subcategory of central skew fields admitting $\ell$ as a subfield and $kV$-submodule. We also classify the subcategory of $\mathscr{C}(\bb{Q})$ of all four-dimensional Galois extensions of $\bb{Q}$ with Galois group $V$ that admit $\ell$ as a subfield.
	\end{abstract}

	\section{Background and setup}
	\label{secIntro}
	
	In this paper, an \emph{algebra} over a field $k$, or more briefly, a $k$-\emph{algebra}, is a $k$-vector space $A$ equipped with a bilinear multiplication map, written by juxtaposition as $(a,b)\mapsto ab$ for $a,b\in A$. Note that we do not require $A$ to be associative. A $k$-algebra $A$ is said to be a \emph{division} \emph{algebra} if $A\neq\{0\}$ and the linear endomorphisms $x\mapsto ax$ and $x\mapsto xa$ of $A$ are bijective for all $a\in A\smpt{0}$.
	
	Given a category $\scr{C}$ such that the collection of isoclasses of $\scr{C}$ forms a set, an interesting problem is to \emph{classify} $\scr{C}$. In this paper, a \emph{classification} of $\scr{C}$ is understood to be an explicitly given transversal of the isoclasses of $\scr{C}$, that is, an explicitly given exhaustive and irredundant set of representatives. A challenging problem, open in most cases, is trying to classify categories of division algebras; for classifications of particular categories of division algebras, see e.g. \cite{darpo2011classification}, \cite{dieterich1998klassifikation}, \cite{dieterich2005classification}.
	
	Let $k$ be a field and let $V$ be Klein's four group. Let $\mathscr{C}(k)$ be the category of four-dimensional unital division $k$-algebras $A$ such that $V\subset \Aut{k}{A}$ and \[N_r(A):=\{n\in A\mid (ab)n=a(bn),\forall a,b\in A\},\]
	the \emph{right nucleus} of $A$, 
	is non-trivial, meaning $k\subsetneq N$. The morphisms in $\mathscr{C}(k)$ are the non-zero algebra morphisms. These are injective due to the division property, hence bijective, as they are endomorphisms of a finite-dimensional vector space. It follows that $\mathscr{C}(k)$ is a \emph{groupoid}, i.e. a category in which all morphisms are isomorphisms. Note that $V\subset \Aut{k}{A}$ gives $A$ the structure of a $kV$-module.
	
	For a given field $k$, let $\scr{Q}(k)$ denote the category of all field extensions $k\subset\ell$ with $[\ell:k]=2$, in which morphisms are the $k$-linear field morphisms. Hereafter a field $k$ with $\Char k\neq 2$ and $\ell\in\scr{Q}(k)$ are fixed. Let $\overline{x}$ denote the Galois conjugate of $x\in\ell$ and let $n_{\ell/k}(x)=x\overline{x}$ be the norm of the field extension $k\subset \ell$. 
	
	For $c:=(c_1,c_2,c_3)\in k^3$ and $x,y\in \ell$ let \begin{equation}
	\label{eqnM(x,y)}
	M_{c}(x,y)=\begin{pmatrix}x & c_2y+c_3\overline{y}\\y & (1-c_1)x+c_1\overline{x}\end{pmatrix}\in \ell^{2\times 2}
	\end{equation}
	and construct the four-dimensional $k$-algebra $A(\ell,c)$ by $A(\ell,c)=\ell^2$ as a vector space over $k$ and multiplication given by
	\begin{equation}
		\label{eqnMultA}
		\begin{pmatrix}x\\y\end{pmatrix}\begin{pmatrix}z\\w\end{pmatrix}=M_{c}(x,y)\begin{pmatrix}z\\w\end{pmatrix}
	\end{equation}
	for $x,y,z,w\in \ell$, where the right hand side of (\ref{eqnMultA}) is ordinary multiplication of $\ell$-matrices. For the remainder of this paper, every occurrence of the word "construction" refers to the above construction of $A(\ell,c)$.
 	Define $C_\ell=\{c\in k^3\mid A(\ell,c) \text{ is a division algebra}\}$ and call $C_\ell$ the set of \emph{admissible triples}. Equivalently, $c$ is admissible if and only if $\det M_{\ell,c}(x,y)\neq 0$ for all $(x,y)\in \ell^2\smpt{(0,0)}$.
	
	For $\gamma\in \Aut{k}{A}$ and $\lambda\in k$, let $E_\lambda(\gamma)=\{x\in A\mid \gamma(x)=\lambda x\}$.
	As in \cite{dieterich2017onFour}, let 
	\[\mathscr{C}_\ell(k)=\{A\in \mathscr{C}(k)\mid \exists n\subset A,\exists\alpha,\beta\in \Aut{k}{A}:\ell \cong n=E_{1}(\beta)\subset N_r(A),\inner{\alpha}{\beta}\cong V \}.\]
	From \cite[Prop. 4.1]{dieterich2017onFour} it is seen that
	\[\mathscr{C}_\ell(k)=\{A\in \mathscr{C}(k)\mid \exists n\subset A, n\cong \ell \text{ and } n\text{ is a }kV\text{-submodule of }A\}.\] By \cite[Theorem 4.5]{dieterich2017onFour} each $A\in\mathscr{C}_\ell(k)$ is isomorphic to $A(\ell,c)$ for some $c\in C_\ell$. 
	
	Due to \cite[Prop. 2.2]{dieterich2017onFour} the category $\mathscr{C}(k)$ decomposes as a coproduct $\mathscr{C}(k)=\mathscr{K}(k)\amalg\mathscr{S}(k)\amalg\mathscr{N}(k)$ of the full subcategories
	$\mathscr{K}(k)$, $\mathscr{S}(k)$, $\mathscr{N}(k)$ consisting of field extensions, skew fields that are central over $k$ and algebras with right nucleus of dimension $2$, respectively. Note that the algebras in $\mathscr{N}(k)$ are not associative.
	If $\scr{L}$ is a classifying list of $\scr{Q}(k)$ we have, from \cite[Prop 4.8]{dieterich2017onFour}, that
	\begin{equation}
	\label{eqnCovering}
	\scr{C}(k)=\left(\bigcup_{\ell\in\scr{L}}\scr{K}_\ell(k)\right)\amalg\left(\bigcup_{\ell\in\scr{L}}\scr{S}_\ell(k)\right)\amalg \left(\coprod_{\ell\in\scr{L}}\scr{N}_\ell(k)\right)
	\end{equation}
	where $\scr{K}_\ell(k)$, $\scr{S}_\ell(k)$, $\scr{N}_\ell(k)$ will be called \emph{local subcategories} and are defined as intersections of $\scr{C}_\ell(k)$ with the corresponding categories $\mathscr{K}(k)$, $\mathscr{S}(k)$, $\mathscr{N}(k)$. 
	
	Given a ring $R$ and $S\subset R$ set $S^*=S\smpt{0}$ and $S_{\mathrm{sq}}=\{s^2\mid s\in S\}$. Given a group $G$ with identity $e$ we set $G^\circ=G\smpt{e}$ and call $G^\circ$ a \emph{punctured group}. Combining \cite[Lemma 3.4, Corollary. 5.4]{dieterich2017onFour} we reproduce the following result.
	\newpage
	\begin{prop}
	\label{bigProp}
	Let $k$ be a field with $\Char k\neq 2$, $\ell\in\scr{Q}(k)$ and $c,d\in C_{\ell}$.
	\begin{itemize}
	\item[\emph{(i)}] We have $A(\ell,c)\cong A(\ell,d)$ in $\scr{C}(k)$ if and only if there exists $x\in \ell^*$ such that $(c_1,c_2,c_3)=(d_1,x^2d_2,n_{\ell/k}(x)d_3)$.
	\item[\emph{(ii)}] If $S(\ell)\subset k^*$ is a transversal of $(k^*/(k^*\cap \ell_{\mathrm{sq}}))^\circ$ then $\{A(\ell,(0,s,0))\mid s\in S(\ell)\}$ classifies $\scr{K}_\ell(k)$.
	\item[\emph{(iii)}] If $T(\ell)\subset k^*$ is a transversal of $(k^*/n_{\ell/k}(\ell^*))^\circ$ then $\{A(\ell,(1,0,t))\mid t\in T(\ell)\}$ classifies $\scr{S}_\ell(k)$.
	\end{itemize}
	\end{prop}
	
	Part (i) of Proposition \ref{bigProp} is an isomorphism condition of algebras constructed from admissible triples. Parts (ii), (iii) of the proposition provide reductions of the \emph{local} classification problems in $\scr{C}(k)$ concerning fields and skew fields to finding transversals of punctured quotient groups of $k$.
	
	In \cite[Section 5]{dieterich2017onFour} the category $\scr{C}(k)$ is classified in the case of $k$ being a finite field of odd order or $k$ an ordered field with $k_\mathrm{sq}=k_{\geq 0}$. In this paper the case $k=\bb{Q}$ is discussed.
	
	\section{On the classification of $\scr{C}(\bb{Q})$}
	
	In this section we investigate the category $\scr{C}(\bb{Q})$ by studying the local parts of the covering (\ref{eqnCovering}). In Section \ref{secK} the object of study is $\scr{K}(\bb{Q})$ and in Proposition \ref{propFields} we classify $\scr{K}(\ell/\bb{Q})$ for each $\ell$ in a classifying list $\scr{L}$ (to be specified below) of $\scr{Q}(\bb{Q})$. The classifications are found by explicitly presenting transversals as in Proposition \ref{bigProp} (ii). 
	
	In Section \ref{secS} we study $\scr{S}(\bb{Q})$ and we construct an exhaustive list of $(\bb{Q}^*/n_{\ell/\bb{Q}}(\bb{Q}^*))^\circ$ for each $\ell\in\scr{L}$. We provide a method for reducing the exhaustive list to a transversal and we hence approach a classification of $\scr{S}_\ell(\bb{Q})$ in view of Proposition \ref{bigProp} (iii). In Proposition \ref{propSkewFields2} we classify $\scr{S}_\ell(\bb{Q})$ in the special case $\ell=\bb{Q}(i)$. The results in this section rely on classic number theoretic theorems dating back to Fermat and Legendre.
	
	In Section \ref{secN} we study $\scr{N}(\bb{Q})$ and present for each $\ell\in\scr{L}$ an irredundant two-parameter family $\tilde{P}(\ell)\subset C_\ell$ of admissible triples such that $A(\ell,c)\in\scr{N}_\ell(k)$ for each $c\in \tilde{P}(\ell)$. In the special case $\ell=\bb{Q}(i)$ we enlarge the two-parameter family to an irredundant four-parameter family. 
	
	\textbf{Notation}: Besides the notation introduced just before Proposition \ref{bigProp} we also set $\bb{N}=\bb{Z}_{\geq0}$ and define $N_2=\bb{N}_{\mathrm{sq}}+\bb{N}_{\mathrm{sq}}$, $Q_2=\bb{Q}_\mathrm{sq}+\bb{Q}_\mathrm{sq}$, i.e. the natural and rational numbers expressible as a sum of two squares, respectively. Let $\bb{P}\subset\bb{N}$ denote the set of prime numbers and let $\bb{P}_1,\bb{P}_3$ denote the subsets of prime numbers $\equiv 1,3\pmod{4}$ respectively. Let $m_p(n)=\max\{k\in\bb{N}: p^k\mid n\}$ for $p\in\bb{P}$, $n\in\bb{Z}\smpt{0}$. In addition, we adopt the convention that $m_p(0)=\infty$ for every $p\in \bb{P}$.
	
	Let $Z=\{n\in\bb{Z}\mid m_p(n)\leq 1\text{ for all } p\in\bb{P}\}\smpt{1}$ so that $Z$ is the set of all square-free integers except $1$. Then, $Z$ is a transversal of $(\bb{Q}^*/\bb{Q}_{\mathrm{sq}}^*)^\circ$ and the list $\scr{L}:=\{\bb{Q}(\sqrt{z})\mid z\in Z\}$ classifies $\scr{Q}(\bb{Q})$. 
	
	\subsection{On $\scr{K}(\bb{Q})$}
	\label{secK}
	
	For each $\ell\in\scr{L}$, transversals $S(\ell)$ as in Proposition \ref{bigProp} (ii) are given as follows.
	
	\begin{prop}
	\label{propFields} Take $z\in Z$ and set $\ell=\bb{Q}(\sqrt{z}) \in\scr{L}$.
	\begin{itemize}
	\item[\emph{(i)}] If $z=-1$, then $S(\ell)=Z_{<0}\smpt{-1}$ is a transversal of $(\bb{Q}^*/(\bb{Q}^*\cap \ell_{\mathrm{sq}}))^\circ$.
	
	\item[\emph{(ii)}] If $z\neq -1$, then $S(\ell)=\left\{w\in Z: |\gcd(z,w)|<\sqrt{|z|}\right\}$ is a transversal of 
	
	$(\bb{Q}^*/(\bb{Q}^*\cap \ell_{\mathrm{sq}}))^\circ$.
	\end{itemize}
	\end{prop}
	
	\begin{proof}
	Let $\pi\colon\bb{Q}^*\longrightarrow (\bb{Q}^*/(\bb{Q}^*\cap \ell_{\mathrm{sq}}))^\circ$ denote the quotient map. We have $\bb{Q}^*\cap \ell_{\mathrm{sq}}=\bb{Q}^*_{\mathrm{sq}}\cup z\bb{Q}^*_{\mathrm{sq}}$. 
	
	(i): Take $q\in \bb{Q}^*$ with $\pi(q)\neq \pi(z)$ and $w\in Z\smpt{z}$ such that $\pi(q)=\pi(w)$. We have $\pi(q)=\pi(\pm w)$ where either $w\in S(\ell)$ or $-w\in S(\ell)$, so $S(\ell)$ exhausts $(\bb{Q}^*/(\bb{Q}^*\cap \ell_{\mathrm{sq}}))^\circ$. Suppose that $w,w'\in S(\ell)$ and $\pi(w)=\pi(w')$. Then, as $w,w'$ have the same sign we have $\tfrac{w}{w'}\in \bb{Q}^*_{\mathrm{sq}}$ so there are integers $r,s$ with $\gcd(r,s)=1$ such that $s^2w=r^2w'$. Take $p\in\bb{P}$ such that $p\mid w$. Then $p\mid r^2w'$ and if $p\mid r$ we get $p^2\mid s^2w$ and since $w$ is square-free we get $p\mid s$, contradiction to $\gcd(r,s)=1$. Thus $p\mid w'$ and by symmetry we get $p\mid w$ if $p\mid w'$, hence $w=w'$. Thus $S(\ell)$ is a transversal of $(\bb{Q}^*/(\bb{Q}^*\cap \ell_{\mathrm{sq}}))^\circ$.
	
	(ii): Suppose now $z\neq -1$ and take $q\in \bb{Q}^*$ with $\pi(q)\neq \pi(z)$. As above, find $w\in Z\smpt{z}$ with $\pi(q)=\pi(w)$ and write $w=dw'$ with $d=\gcd(z,w)>0$. If $d<\sqrt{|z|}$ then $w\in S(\ell)$. Otherwise, $\left|\tfrac{z}{d}\right|<\sqrt{|z|}$. Indeed, if $d\geq \sqrt{|z|}$ and $\left|\tfrac{z}{d}\right|\geq \sqrt{|z|}$ then we must have $d=\left|\tfrac{z}{d}\right|=\sqrt{|z|}$ so that $\sqrt{|z|}$ is an integer, contradiction to $z$ being a square-free integer different from $\pm 1$. Then, $\tfrac{z}{d}w'\in S(\ell)$ and $\pi\left(\tfrac{z}{d}w'\right)=d^2\tfrac{z}{d}w'(\bb{Q}^*_\mathrm{sq}\cup z\bb{Q}^*_\mathrm{sq})=w(\bb{Q}^*_\mathrm{sq}\cup z\bb{Q}^*_\mathrm{sq})=\pi(q)$. Hence, $S(\ell)$ exhausts $(\bb{Q}^*/(\bb{Q}^*\cap \ell_{\mathrm{sq}}))^\circ$.
		
	To prove irredundance of $S(\ell)$, take $w_1,w_2\in  S(\ell)$ and suppose that $\pi(w_1)=\pi(w_2)$. Then, either $\tfrac{w_1}{w_2}\in\bb{Q}^*_{\mathrm{sq}}$ or $\tfrac{w_1}{w_2}\in z\bb{Q}^*_{\mathrm{sq}}$. For $i\in\{1,2\}$, write $w_i=d_iw_i'$ with $d_i=\gcd(w_i,z)>0$ and $|d_i|<\sqrt{|z|}$.
		
	Suppose first that $\tfrac{w_1}{w_2}\in\bb{Q}^*_{\mathrm{sq}}$. Then, there are integers $r,s$ with $\gcd(r,s)=1$ such that $w_1=\tfrac{r^2}{s^2}w_2$ which implies $s^2d_1w_1'=r^2d_2w_2'$. Suppose there is $p\in\bb{P}$ with $p\mid d_1$. Since $d_1\mid z$ and $\gcd(z,w_2')=1$ we have $p\mid r^2d_2$. If $p\mid r$ then $p^2\mid r^2$ and so $p^2\mid s^2d_1w_1'=s^2w_1$, contradiction to $w_1$ being square-free. Hence, $p\mid d_2$. By symmetry, $p\mid d_2$ implies $p\mid d_1$ and hence $d_1=d_2$. Thus, $s^2w_1'=r^2w_2'$ and by a similar argument, we get $w_1'=w_2'$ and hence also $w_1=w_2$ as desired.
		
	Suppose now $\tfrac{w_1}{w_2}\in z\bb{Q}^*_{\mathrm{sq}}$ so that we can write 
	$s^2d_1w_1'=r^2zd_2w_2'$ for some relatively prime integers $r,s$. Using $\gcd(z,w_i')=1$ and arguments as above, we get $w_1'=\pm w_2'$ and hence $s^2d_1=\pm r^2zd_2$. Thus, $\tfrac{r^2}{s^2}=\pm\tfrac{z}{d_1}\cdot d_2\in\bb{N}$. Hence, $\pm\tfrac{z}{d_1}\cdot d_2$ is the square of an integer, and as $\tfrac{z}{d_1}$, $d_2$ are both square-free we must have $\pm\tfrac{z}{d_1}=d_2$ so that $|z|=d_1d_2$, which is a contradiction since $d_1,d_2<\sqrt{|z|}$ by assumption.
	\end{proof}
	
	\begin{rem}
	By Proposition \ref{bigProp} (ii) we arrive, for each $\ell\in\scr{L}$, at a classification of the category $\scr{K}_\ell(\bb{Q})$, which consists of all $4$-dimensional Galois extensions of $\bb{Q}$ with Galois group $V$ that contain $\ell$ as a subfield.
	\end{rem}
	
	\subsection{On $\scr{S}(\bb{Q})$}
	\label{secS}
	
	Given a field $\ell=\bb{Q}(\sqrt{z})\in\scr{L}$ we attempt to find a transversal $T(\ell)$ as in Proposition \ref{bigProp} (iii). The set $Z\setminus n_{\ell/\bb{Q}}(\ell^*)$ exhausts $(\bb{Q}^*/n_{\ell/\bb{Q}}(\ell^*))^\circ$
	and in order to present this set explicitly we have the following result.
	
	\begin{prop}
	\label{propSkewFields1}
	Take $z\in Z$ and set $\ell=\bb{Q}(\sqrt{z})$. Take $w\in Z$ and set $0<d=\gcd(z,w)$. Then, $w\in n_{\ell/\bb{Q}}(\ell^*)$ if and only if all of the following four conditions are satisfied: $z,w$ are not both negative, $z$ is a square mod $\tfrac{w}{d}$, $w$ is a square mod $\tfrac{z}{d}$ and $-\tfrac{zw}{d^2}$ is a square mod $d$.
	\end{prop}
	
	\begin{proof}
	We have that $w\in n_{\ell/\bb{Q}}(\ell^*)$ if and only if there exists $(a,b)\in \bb{Q}^2\smpt{(0,0)}$ such that $w=a^2-zb^2$. We claim that the following are equivalent:
	\begin{enumerate}[(i)]
	\item There exists $(a,b)\in \bb{Q}^2\smpt{(0,0)}$ such that $w=a^2-zb^2$.
	
	\item There exists $(a,b,c)\in\bb{Z}^3\smpt{(0,0,0)}$ such that $a^2-wb^2-zc^2=0$.
	
	\item There exists $(a,b,c)\in\bb{Z}^3\smpt{(0,0,0)}$ such that $da^2-\tfrac{w}{d}b^2-\tfrac{z}{d}c^2=0$.
	\end{enumerate}
	
	Indeed, suppose $(a,b)$ satisfies (i). If $a=0$ then $\frac{z}{w}\in \bb{Q}_{\mathrm{sq}}$ then it is immediate that (ii) holds. Otherwise, write $b=\tfrac{b_1}{b_2}$, $a=\tfrac{a_1}{a_2}$, $a_1\neq 0$. Then $a_1^2-zb_1^2-w(a_2b_2)^2=0$, so (ii) holds. Conversely, if $(a,b,c)$ satisfies (ii), then $b\neq 0$, for otherwise $c\neq 0$ and $z=\left(\tfrac{a}{c}\right)^2$, contradicting $z$ square-free. With $b\neq 0$ we see that $w=\left(\tfrac{a}{b}\right)^2-z\left(\tfrac{c}{b}\right)^2$, so (i) holds.
	
	Suppose $(a,b,c)$ satisfies (ii). Since $d\mid z,w$ we get $d\mid a^2$ which implies $d\mid a$ as $d$ is square-free. Write $a=da'$ so that $d(a')^2-\tfrac{w}{d}b^2-\tfrac{z}{d}c^2=0$ shows that (iii) holds. Finally, if (iii) holds with $(a,b,c)$, multiply $da^2-\tfrac{w}{d}b^2-\tfrac{z}{d}c^2=0$ with $d$ to see that (ii) holds. 
	
	Now, by Legendre's theorem on ternary integral quadratic forms \cite[Proposition 5.11]{niven1991number} item (iii) above is equivalent to the conditions in the statement of the proposition.
	\end{proof}
	
	\begin{rem}
	Proposition \ref{propSkewFields1} allows us to present \emph{explicitly} the set $Z\setminus n_{\ell/\bb{Q}}(\ell^*)$ which exhausts $(\bb{Q}^*/n_{\ell/\bb{Q}}(\ell^*))^\circ$. However, in order to get a transversal as in Proposition \ref{bigProp} (iii) we need an exhausting set which also is irredundant. 
	
	Take $w,w'\in Z\setminus n_{\ell/\bb{Q}}(\ell^*)$. Then, $w,w'$ represent the same coset in $(\bb{Q}^*/n_{\ell/\bb{Q}}(\ell^*))^\circ$ and only if $\tfrac{w}{w'}\in n_{\ell/\bb{Q}}(\ell^*)$ which holds if and only if $\tfrac{(w')^2}{d^2}\cdot\tfrac{w}{w'}=\tfrac{w'w}{d^2}\in n_{\ell/\bb{Q}}(\ell^*)$. Now, $\tfrac{w'w}{d^2}\in Z$ and we can use Proposition \ref{propSkewFields1} to answer the question whether $\tfrac{w'w}{d^2}\in n_{\ell/\bb{Q}}(\ell^*)$. Thus, we can use Proposition \ref{propSkewFields1} in order to successively discard elements of $Z\setminus n_{\ell/\bb{Q}}(\ell^*)$ that prevent irredundance, hence we have an "infinite method" of reducing $Z\setminus n_{\ell/\bb{Q}}(\ell^*)$ to a transversal of $(\bb{Q}^*/n_{\ell/\bb{Q}}(\ell^*))^\circ$. However, more work is required in order to present these transversals explicitly.
	\end{rem}
	
	We now consider a special case, in which we can explicitly present a transversal of $(\bb{Q}^*/n_{\ell/\bb{Q}}(\ell^*))^\circ$.
	
	\begin{prop}
	\label{propSkewFields2}
	Let $\ell=\bb{Q}(i)$ and define $T(\ell)=\{z\in Z\mid m_p(z)>0\emph{ implies }p\in\bb{P}_3\}$. Then $T(\ell)$ is a transversal of $(\bb{Q}^*/n_{\ell/\bb{Q}}(\ell^*))^\circ$.
	\end{prop}
	
	\begin{proof}
	In the present case, $n_{\ell/\bb{Q}}(\ell^*)=Q_2^*$. Let $\pi:\bb{Q}^*\longrightarrow\bb{Q}^*/Q_2^*$ denote the quotient map. The following characterization of elements of $Q_2$ is stated as an exercise in \cite[p. 352]{sierp1964numbers}: $\tfrac{a}{b}\in Q_2$ if and only if $ab\in N_2$. This characterization is a straightforward consequence of Fermat's characterization of elements of $N_2$: Given $n\in\bb{Z}_{>0}$ we have $n\in N_2$ if and only if $m_p(n)$ is even for all $p\in\bb{P}_3$ (cf. \cite[Theorem 2.15]{niven1991number}).
	
	Take $q\in\bb{Q}^*$ such that $\pi(q)$ is not the identity in $\bb{Q}^*/Q_2^*$. Since $\bb{Q}^*_\mathrm{sq}\subset Q_2^*$ there is $z\in Z$ such that $\pi(q)=\pi(z)$ and since $p\in N_2\subset Q_2$ for each $p\in \bb{P}_1$ can even take $z\in T(\ell)$. Thus $T(\ell)$ exhausts $\bb{Q}^*/Q_2^*$.
	
	Take now $z,z'\in T(\ell)$ and suppose $\pi(z)=\pi(z')$. Then, $\tfrac{z}{z'}\in Q_2^*$, so $zz'\in N_2$. Since $z,z'$ are square-free integers only divisible by primes in $\bb{P}_3$ we must have $z=z'$ and $T(\ell)$ is a transversal of $\bb{Q}^*/Q_2^*$ as desired.
	\end{proof}
	
	\begin{rem} By \cite[Corollary 4.6]{dieterich2017onFour} $\scr{S}_\ell(\bb{Q})$ coincides with the category of all four-dimensional Hurwitz division algebras over $\bb{Q}$ that contain $\ell$ as a subfield and $\bb{Q}V$-submodule. Four dimensional Hurwitz algebras are also known as four-dimensional unital composition algebras or quaternion algebras. Proposition \ref{propSkewFields2} yields a classification of the four-dimensional rational Hurwitz algebras that admit $\bb{Q}(i)$ as a subfield and $\bb{Q}V$-submodule.
	\end{rem}
	
	\subsection{On $\scr{N}(\bb{Q})$}
	\label{secN}
	
	In this section we produce families of algebras in $\scr{N}_\ell(\bb{Q})$ for each $\ell\in\scr{L}$ by producing families of admissible triples to which the construction from Section \ref{secIntro} is applied. Considering the definition of $C_\ell$, equations (\ref{eqnM(x,y)}), (\ref{eqnMultA}), and $\ell=\bb{Q}(\sqrt{z})$, we have for general $c\in\bb{Q}^3$ that $c\in C_\ell$ if and only if the system
	\begin{equation}
	\label{eqnSystem}
	\begin{cases}
	0=x_1^2+z(1-2c_1)x_2^2-(c_2+c_3)y_1^2+z(c_3-c_2)y_2^2\\
	(1-c_1)x_1x_2=c_2y_1y_2
	\end{cases}
	\end{equation}
	admits the trivial solution $(x_1,x_2,y_1,y_2)\in \bb{Q}^4$ only. Since this system is homogeneous, the existence of a non-trivial rational solution is equivalent to the existence of a primitive integral solution.
	
	In the following proposition we produce three-parameter families of triples in $C_\ell$ by choosing $c$ so that the right hand side of the first equation in (\ref{eqnSystem}) is always non-negative.
	
	\begin{prop}
	\label{propNonAss1}
	Take $z\in Z$ and set $\ell=\bb{Q}(\sqrt{z})$.
	\begin{enumerate}
	\item[\emph{(i)}] If $z>0$ then 
	$P(\ell):=\{(c_1,c_2,c_3)\mid c_1<\tfrac{1}{2}\,\wedge\,c_2<0\,\wedge\,c_2<c_3<-c_2\}\subset C_\ell$.
	
	\item[\emph{(ii)}] If $z<0$ then $P(\ell):=\{(c_1,c_2,c_3)\mid c_1>\tfrac{1}{2}\,\wedge\,c_3<0\,\wedge\,c_3<c_2<-c_3\}\subset C_\ell$.
	\end{enumerate}
	\end{prop}
	
	\begin{proof}
	For (i), $c_1<\frac{1}{2}$ ensures that $z(1-2c_1)>0$. From $c_3>c_2$ we get $z(c_3-c_2)>0$ and from $c_3<-c_2$ we get $c_2+c_3<0$ and hence $-(c_2+c_3)>0$ so the right hand side of the first equation of (\ref{eqnSystem}) is non-negative for each $(x_1,x_2,y_1,y_2)\in\bb{Q}^4$ and hence (\ref{eqnSystem}) only admits the trivial solution, implying $(c_1,c_2,c_3)\in C_\ell$. Similar verifications yield (ii).
	\end{proof}
	
	In the following corollary we extract from each three-parameter family $P(\ell)$ of the previous proposition a two-parameter irredundant subfamily $\tilde{P}(\ell)\subset P(\ell)$.
	
	\begin{cor}
	\label{corNonAss1}
	Take $z\in Z$, set $\ell=\bb{Q}(\sqrt{z})$ and let $S(\ell)\subset \bb{Q}^*$ be as in \emph{Proposition \ref{propFields}}, $P(\ell)$ as in \emph{Proposition \ref{propNonAss1}}.
	\begin{enumerate}
		\item[\emph{(i)}] If $z>0$ then $\tilde{P}(\ell):=\{(r,s,\tfrac{1}{2})\mid r<\tfrac{1}{2}\,\wedge\,s\in S(\ell)_{<0}\}\subset C_\ell$ is an irredundant subset of $P(\ell)$.
		
		\item[\emph{(ii)}] If $z<0$ then $\tilde{P}(\ell):=\{(r,s,s-1)\mid r>\tfrac{1}{2}\,\wedge\,s\in S(\ell)_{<0}\}\subset C_\ell$ is an irredundant subset of $P(\ell)$.
	\end{enumerate}
	\end{cor}
	
	\begin{proof}
	This corollary is straightforward consequence of Proposition \ref{bigProp} (i), the definition of $P(\ell)$ and the fact that $S(\ell)$ is a transversal of $(\bb{Q}^*/(\bb{Q}^*\cap \ell_\mathrm{sq}))^\circ$.
	\end{proof}
	
	\begin{rem}
	Given $c=(c_1,c_2,c_3)\in C_\ell$ we have from \cite[Prop 3.3]{dieterich2017onFour} that $A(\ell,c)\in \scr{K}(\bb{Q})$ if and only if $c_1=c_3=0$ and $A(\ell,c)\in \scr{S}(\bb{Q})$ if and only if $(c_1,c_2)=(1,0)$. Thus, for each $c\in \tilde{P}(\ell)$ as above we have $A(\ell,c)\in\scr{N}_\ell(\bb{Q})$.
	\end{rem}
	
	We now show how further families of triples in $C_\ell$ can be found for particular $\ell$ by choosing $c$ such that various number theoretic obstructions guarantee that the system (\ref{eqnSystem}) only admits the trivial solution.
	
	\begin{prop}
	\label{propNonAss2}
	Take $z\in Z$ such that $m_p(z)>0$ for some $p\in\bb{P}_3$, and fix such $p$. Then, if $c=(c_1,c_2,c_3)\in\bb{Z}^3$ satisfies $1-2c_1\equiv 1\pmod p$, $c_2\equiv -1\pmod p$, $c_3\equiv 0\pmod p$ we have $c\in C_\ell$.
	\end{prop}
	
	\begin{proof}
	Suppose $c$ is as above and, towards a contradiction, that (\ref{eqnSystem}) admits a primitive integral solution $(x_1,x_2,y_1,y_2)$. Considering the first equation of (\ref{eqnSystem}) modulo $p$ we get $x_1^2+y_2^2\equiv 0\pmod p$. Since $p\in\bb{P}_3$ we get, by \cite[Lemma 2.14]{niven1991number} that $p\mid x_1,x_2$. Thus, \[p^2\mid x_1^2-(c_2+c_3)y_1^2=z((c_2-c_3)y_2^2-(1-2c_1)x_2^2).\] Since $z$ is square-free, we get $p\mid (c_2-c_3)y_2^2-(1-2c_1)x_2^2$. By assumptions on $c_1,c_2,c_3$ we get $y_2^2+x_2^2\equiv 0\pmod p$ so $p\mid x_2,y_2$, contradiction to $(x_1,x_2,y_1,y_2)$ being a primitive solution of (\ref{eqnSystem}).
	\end{proof}
	
	\begin{rem}
	Note that, for the same reason as stated in the remark after Corollary \ref{corNonAss1}, for all triples $c$ produced in Proposition \ref{propNonAss2} we have $A(\ell,c)\in\scr{N}_\ell(\bb{Q})$.
	\end{rem}
	
	In the remaining part of this section we investigate the case $\ell=\bb{Q}(i)$ further.
	
	\begin{prop}
	\label{propNonAss3}
	\begin{itemize}
	\item[\emph{(i)}] The one-parameter family $P_1(\bb{Q}(i)):=\left.\left\{\left(\tfrac{1-q}{2},0,-1\right)\right|q\in \bb{Q}_{>0}\setminus Q_2\right\}$ satisfies $P_1(\bb{Q}(i))\subset C_{\bb{Q}(i)}$ and is irredundant.
	
	\item[\emph{(ii)}] The one-parameter family \[P_2(\bb{Q}(i)):=\{(1,n,0)\mid n\in Z_{<0}\emph{ such that there exists } p\in\bb{P}_3\emph{ with } m_p(n)=1\}\] satisfies $P_2(\bb{Q}(i))\subset C_{\bb{Q}(i)}$ and is irredundant.
	\end{itemize}
	
	\end{prop}
	
	\begin{proof}
	(i): Let $q\in\bb{Q}_{>0}\setminus Q_2$ and set $(c_1,c_2,c_3)=\left(\tfrac{1-q}{2},0,-1\right)$. Then the system (\ref{eqnSystem}) looks as follows
	\begin{equation}
		\label{eqnSystemNonAss3}
		\begin{cases}
		0=x_1^2-qx_2^2+y_1^2+y_2^2\\
		\left(\tfrac{q+1}{2}\right)x_1x_2=0
		\end{cases}
	\end{equation}
	and the second equation implies that a solution $(x_1,x_2,y_1,y_2)\in \bb{Q}^4$ must satisfy $x_1=0$ or $x_2=0$.
	
	If $x_2=0$ then the first equation in (\ref{eqnSystemNonAss3}) is $0=x_1^2+y_1^2+y_2^2$ which implies $x_1=y_1=y_2=0$. If $x_1=0$ then the first equation becomes $qx_2^2=y_1^2+y_2^2$, which implies $q\in Q_2$ unless $x_2=0$, which was seen to imply that $y_1=y_2=0$. Thus, the system (\ref{eqnSystemNonAss3}) only admits the trivial solution and we conclude that $P_1(\bb{Q}(i))\subset C_{\bb{Q}(i)}$.
	
	For irredundance, we merely note that the elements of $P_1(\bb{Q}(i))$ all have different first components whence the claim follows immediately from Proposition \ref{bigProp} (i).
	
	(ii): For $(c_1,c_2,c_3)=\left(1,n,0\right)$ the system (\ref{eqnSystem}) becomes
			\begin{equation}
				\label{eqnSystemNonAss4}
				\begin{cases}
				0=x_1^2+x_2^2-ny_1^2+ny_2^2\\
				0=ny_1y_2
				\end{cases}
			\end{equation}
		and we assume that (\ref{eqnSystemNonAss4}) admits a primitive integral solution $(x_1,x_2,y_1,y_2)$. Then, by the second equation, we have $y_1=0$ or $y_2=0$. If $y_2=0$ then the first equation implies $x_1=x_2=y_1=0$. If $y_1=0$, then the first equation implies $-ny_2^2=x_1^2+x_2^2$. Fix $p\in \bb{P}_3$ such that $m_p(n)=1$. Then, $p\mid x_1^2+x_2^2$ and hence $m_p(x_1^2+x_2^2)$ is even. By assumption of $(x_1,x_2,y_1,y_2)$ being primitive, we have $p\nmid y_2$. Thus, $m_p(-ny_2^2)$ is equal to $1$ while $m_p(x_1^2+x_2^2)$ is even, contradiction, and we conclude that $P_2(\bb{Q}(i))\subset C_{\bb{Q}(i)}$. 
		
		For irredundance, we note that the middle component runs through a subset of the transversal $S(\bb{Q}(i))$ of $(\bb{Q}^*/(\bb{Q}_{\mathrm{sq}}^*\cap \bb{Q}^*))^\circ$ given in Proposition \ref{propFields}.
	\end{proof}
	
	Considering the first components of elements in $\tilde{P}(\bb{Q}(i))$, $P_1(\bb{Q}(i))$ and $P_2(\bb{Q}(i))$ (the families obtained in Corollary \ref{corNonAss1} and Proposition \ref{propNonAss3}) we conclude with Proposition \ref{bigProp} that $\tilde{P}(\bb{Q}(i))\cup P_1(\bb{Q}(i))$ and $P_1(\bb{Q}(i))\cup P_2(\bb{Q}(i))$ are irredundant subsets of $C_{\bb{Q}(i)}$. Since the third component of a given element in $P_2(\bb{Q}(i))$ is $0$ and the third component of a given element in $\tilde{P}(\bb{Q}(i))$ is non-zero we conclude with Proposition \ref{bigProp} that $\tilde{P}(\bb{Q}(i))\cup P_2(\bb{Q}(i))$ is irredundant and hence the family defined by $F(\bb{Q}(i))=\tilde{P}(\bb{Q}(i))\cup P_1(\bb{Q}(i))\cup P_2(\bb{Q}(i))$ is an irredundant four-parameter family of admissible triples in $C_{\bb{Q}(i)}$. Furthermore, $A(\bb{Q}(i),c)\in\scr{N}_{\bb{Q}(i)}(\bb{Q})$ for each $c\in F(\bb{Q}(i))$.
	
	In this section, we studied $\scr{N}(\bb{Q})$ locally, i.e. via the components of its covering by local subcategories in (\ref{eqnCovering}) indexed by $\scr{L}$. For each field $\ell\in\scr{L}$ we studied the system (\ref{eqnSystem}) with the production of admissible triples $c\in C_\ell$ such that $A(\ell,c)\in \scr{N}(\ell/\bb{Q})$ in mind. The nature of $\bb{Q}$ gives rise to two complexities: since $|\scr{L}|=\infty$ there is an abundance of local subcategories to consider and since $\bb{Q}$ is far from being algebraically closed, there should be many possibilites of $c$ such that (\ref{eqnSystem}) only admits the trivial solution. By varied approaches from elementary number theory we produced parametrized families of triples $c\in C_\ell$ such that $A(\ell,c)\in \scr{N}_\ell(\bb{Q})$ for each $\ell$, establishing richness of $\scr{N}_\ell(\bb{Q})$, but a classification of $\scr{N}_\ell(\bb{Q})$ seems distant.
	\bibliographystyle{siam}
	\bibliography{bibl}
\end{document}